\documentclass[a4paper,reqno,11pt]{amsart}
\usepackage{amssymb, amsfonts}
\usepackage{fancyhdr, mathrsfs, graphicx, epsfig, amsmath, amsthm}
\usepackage{float}
\usepackage{enumerate}
\usepackage{amscd}
\usepackage{tikz, multicol}
\usepackage[colorlinks,linkcolor=blue,anchorcolor=blue,citecolor=red]{hyperref}
\usetikzlibrary{arrows,decorations.pathmorphing}
\usetikzlibrary{chains}

\theoremstyle{plain}
 \newtheorem{theorem}{Theorem}[section]
 \newtheorem{proposition}[theorem]{Proposition}
 \newtheorem{lemma}[theorem]{Lemma}
 
 \newtheorem{question}[theorem]{Question}
 \newtheorem{open problem}[theorem]{Open problem}
 \newtheorem{proposed project}[theorem]{Proposed project}
\theoremstyle{definition}
 
 \newtheorem{definition}[theorem]{Definition}
\theoremstyle{remark}
 \newtheorem{remark}[theorem]{Remark}
 
 \numberwithin{equation}{section}

\title{The moduli space of the tropicalizations of Riemann surfaces}
\author{Dali Shen}
\date{\today}
\address{School of Mathematics, Tata Institute of Fundamental Research, Homi Bhabha Road, Mumbai 400005, India}
\email{dalishen@math.tifr.res.in}

\begin{document}

\setcounter{page}{1}
\pagenumbering{arabic}

\maketitle

\begin{abstract}
In this paper we study the moduli space of the tropicalizations of Riemann surfaces. 
We first tropicalize a smooth pointed Riemann surface by a graph defined by its (hyperbolic) pair of pants decomposition. 
Then we can construct the moduli space of tropicalizations based on a fixed regular tropicalization, and 
compactify it by adding strata parametrizing weighted contractions. 
We show that this compact moduli space is also Hausdorff. In the end, we compare this moduli space with the moduli 
space of Riemann surfaces, establishing a partial order-preserving correspondence between the 
stratifications of these two moduli spaces.
\end{abstract}

\tableofcontents

\section{Introduction}
We wish to construct the moduli space of the tropicalizations of Riemann surfaces with marked points in this paper,  
so that a partial order-preserving correspondence can be established between the stratifications of this moduli space 
and the moduli space of Riemann surfaces.

The tropical methods have already been used to study the moduli theory of algebraic curves during the past decade, 
via the help of non-Archimedean geometry. In \cite{Caporaso-2013}, Caporaso constructed the moduli space of tropical curves
with marked points. And later in \cite{Abramovich-Caporaso-Payne}, Abramovich, Caporaso and Payne identified this space 
with the skeleton of the moduli stack of stable curves, via the analytification of its coarse moduli space. Hence they 
established an order-reversing correspondence between the stratifications of the two moduli spaces. For their purpose 
they tropicalize an algebraic curve by its dual graph and interpret the edge-length as the ``complexity'' of the node. 
These work are also summarized in the expository papers \cite{Abramovich-2013} and \cite{Caporaso-2018}.

However, we want to establish an order-preserving correspondence between stratifications of some moduli space in a tropical 
sense and the moduli space of Riemann surfaces (or complex algebraic curves). 
In order to realize that, we tropicalize a smooth pointed Riemann surface 
by the graph defined by its pair of pants decomposition, and the edge-length is accordingly given by the length 
of the geodesic representative of the corresponding boundary component of the (hyperbolic) pair of pants. 
We denote this tropicalization by $(G,\ell)$. 
The underlying graph $G$ is always trivalent and we call 
such a topological tropicalization a regular tropicalization for a smooth pointed Riemann surface. Since a regular tropicalization 
is not unique, we always fix a regular tropicalization $G$ to start with for constructing our moduli space.
Now each tropicalization $(G,\ell)$ corresponds to a point in the open cone 
\[
\mathcal{C}(G):=\mathbb{R}_{>0}^{|E(G)|}
\]
if to each edge we associate a coordinate. Of course, the automorphism group $\mathrm{Aut}(G)$ of the graph 
induces an action on this cone by permuting the coordinates. So we construct the moduli space 
of tropicalizations $(G,\ell)$ as the following quotient space:
\[
M_{G}:=\mathcal{C}(G)/\mathrm{Aut}(G)=\mathbb{R}_{>0}^{|E(G)|}/\mathrm{Aut}(G)
=\mathbb{R}_{>0}^{3g-3+n}/\mathrm{Aut}(G)
\]
with the quotient topology.

Then a weighted contraction of $G$, for which the genus of the graph is preserved, is introduced to represent the boundary 
point of $\mathcal{C}(G)$. In this sense the regular tropicalization is viewed as a $\underline{0}$-weighted graph 
$(G,\underline{0})$, and those dependent notions are thus modified accordingly. In particular, 
the above moduli space $M_{G}$ will be rewritten as $M^{\mathrm{tr}}(G,\underline{0})$. 
We denote by $(G',w')\preceq (G,w)$ if $(G',w')$ is a weighted contraction of $(G,w)$. 
Without surprise, a weighted contraction of $(G,\underline{0})$ can be interpreted as the tropicalization of a 
nodal Riemann surface which is obtained by shrinking the corresponding geodesic representatives to a point. 
Hence we get a partially compactified space $\mathcal{C}(G,\underline{0})^{+}=\mathbb{R}_{\geq 0}^{|E(G)|}$ of the cone.

However, this space is still not compact since the edge-length is allowed to go to arbitrarily large. So we need to see what it 
would like to be when the edge-length goes to infinity. In fact, if we let the edge-length go to infinity, 
the Riemann surface in question becomes the normalization 
of the nodal Riemann surface obtained by letting the corresponding geodesic representatives go to zero. This fact connects the 
two tropicalizations corresponding respectively to the edge-length going to infinity or zero. 
And on the other hand, the normalized 
Riemann surface is not connected or of the same genus any more for which we do not want to tropicalize. 
So that we can identify $\infty$ to $0$ for each coordinate in order to construct our moduli space.

Therefore, we extend the action of $\mathrm{Aut}(G,\underline{0})$ to the compactified space 
$\overline{\mathcal{C}(G,\underline{0})}=(\mathbb{S}^{1})^{|E(G)|}$ in such a way that 
the quotient space 
\[
\overline{M^{\mathrm{tr}}(G,\underline{0})}:=\overline{\mathcal{C}(G,\underline{0})}/\mathrm{Aut}(G,\underline{0})=
(\mathbb{S}^{1})^{|E(G)|}/\mathrm{Aut}(G,\underline{0})
\]
identifies isomorphic tropicalizations. 
This quotient space endowed with the quotient topology 
is our desired compactified moduli space for $M^{\mathrm{tr}}(G,\underline{0})$.

Then we arrive at the first main result of this paper, some properties of this moduli space.

\begin{theorem}[Theorem \ref{thm:hausdorff}]
	Let $(G,\underline{0})$ be a regular tropicalization of a smooth $n$-pointed genus $g$ Riemann surface. Then we have 
	\begin{enumerate}
		\item There is a stratification for $\overline{M^{\mathrm{tr}}(G,\underline{0})}$ as follows
		\[
		\overline{M^{\mathrm{tr}}(G,\underline{0})}=\bigsqcup_{(G',w') \preceq (G,\underline{0})}M^{\mathrm{tr}}(G',w'),
		\]
		where $M^{\mathrm{tr}}(G,\underline{0})$ is open and dense 
		in $\overline{M^{\mathrm{tr}}(G,\underline{0})}$. 
		
		\item $\overline{M^{\mathrm{tr}}(G,\underline{0})}$ is compact and Hausdorff as a topological space.
	\end{enumerate}
\end{theorem}

The tropicalization of a nodal Riemann surface can be realized as a weighted contraction of a regular tropicalization 
(i.e., a trivalent genus $g$ graph with $n$ leaves), via its normalization. 
In view of this we have as well a stratification for the moduli space of Riemann surfaces $\overline{M}_{g,n}$ as follows:
\[
\overline{M}_{g,n}=\bigsqcup_{\big(\text{$(G'_{\preceq G},w')$, $G$ trivalent, genus $g$, $n$ leaves}\big)/\sim}
M^{\mathrm{rs}}(G'_{\preceq G},w'),
\]
where $M^{\mathrm{rs}}(G'_{\preceq G},w')$ denotes the locus in 
$\overline{M}_{g,n}$ of those nodal Riemann surfaces obtained by shrinking those geodesic representatives, corresponding to 
contracted edges in $(G,\underline{0})$ to get $(G',w')$, to a point; 
and $(G'_{\preceq G_{1}},w')\sim(G''_{\preceq G_{2}},w'')$ 
if they support the Riemann surfaces of the same topological type.

Although the weighted contractions of a fixed 
regular tropicalization can not exhaust all the topological types of nodal Riemann surfaces, 
we still have the following partial partition analogy between $\overline{M^{\mathrm{tr}}(G,\underline{0})}$ and $\overline{M}_{g,n}$.

\begin{theorem}[Theorem \ref{thm:analogy}]
	Fix a regular tropicalization $(G,\underline{0})$ of an $n$-pointed genus $g$ Riemann surface. Let $(G',w')$ be a weighted contraction of 
	it. Then the association as below
	\[
	M^{\mathrm{tr}}(G',w')\mapsto M^{\mathrm{rs}}(G'_{\preceq G},w')
	\]
	gives a map from the stratification of $\overline{M^{\mathrm{tr}}(G,\underline{0})}$ to the stratification of $\overline{M}_{g,n}$.
	And we have 
	\begin{enumerate}
		\item $\dim M^{\mathrm{tr}}(G',w')=\dim M^{\mathrm{rs}}(G'_{\preceq G},w')=|E(G')|$.
		
		\item Suppose $(G'',w'')\preceq (G',w')$, then we have $M^{\mathrm{tr}}(G'',w'')\subset  \overline{M^{\mathrm{tr}}(G',w')}$ 
		and $M^{\mathrm{rs}}(G''_{\preceq G},w'')\subset  \overline{M^{\mathrm{rs}}(G'_{\preceq G},w')}$.
	\end{enumerate}
\end{theorem}

Note that, unfortunately, the above association in the theorem is neither surjective nor injective.

Since our tropicalization is essentially a tropical curve, we largely make use of the techniques on the moduli of tropical 
curves (or equivalently, weighted metric graphs) which are well elaborated in \cite{Caporaso-2013} and \cite{Caporaso-2018}, 
except for some suitable modifications for our situation.

The paper is organized as follows. In Section $\ref{sec:tropicalization}$ we tropicalize a smooth $n$-pointed Riemann surface 
through its (hyperbolic) pair of pants decomposition. In Section $\ref{sec:moduli}$ we construct our moduli space 
of tropicalizations based on a fixed regular tropicalization, and then comapctify it by adding strata parametrizing weighted 
contractions. We can also show some basic properties of this space. In the last section, Section $\ref{sec:comparison}$, 
we compare this space with the moduli space of Riemann surfaces, trying to establish some correspondence 
between stratifications of these two moduli spaces.

\section{Tropicalization by pants decompositions}\label{sec:tropicalization}

This section is devoted to giving a brief introduction to the tropicalization of a Riemann surface with marked points. 
We tropicalize a pointed Riemann surface by using its pair of pants decomposition. In fact, a pants decomposition of 
a closed topological surface implies some important information on the moduli space of Riemann surfaces supported on it, 
e.g., its parameters, dimension and so on.  
A good reference on Riemann surfaces and 
their moduli is Looijenga's notes \cite{Looijenga-2017}, as well as a much bigger and comprehensive volume of 
Arbarello-Cornalba-Griffiths \cite{Arbarello-Cornalba-Griffiths}.

Let $S$ be a closed connected topological surface of genus $g$. It is well-known that a conformal structure with an orientation 
endowed on $S$ turns $S$ into a Riemann surface. Indeed, it coincides with the classical definition of a Riemann surface: the surface 
$S$ admits an atlas of complex charts with holomorphic transition functions. That is because a conformal structure on $S$ with 
an orientation defines a complex structure $J_{p}$ on each tangent space $T_{p}S$ of $S$, and a Riemann metric is locally 
conformal to a Euclidean metric so that the coordinate changes of this atlas are holomorphic.
Therefore, it is not surprising that we can describe the variation of complex structures on $S$ in terms of the 
variation of conformal structures on it, up to some equivalence encoded by a group action. 

Let $P$ be a closed subset of $S$, not necessarily finite, but can be empty. Then all the diffeomorphisms of $S$ leaving 
$P$ pointwise fixed form a group, denoted by $\mathrm{Diff}(S,P)$. We write its identity component as $\mathrm{Diff}^{\circ}(S,P)$, 
which is a normal subgroup of $\mathrm{Diff}(S,P)$. Likewise, those orientation preserving diffeomorphisms of $\mathrm{Diff}(S,P)$ 
form a group as well, denoted by $\mathrm{Diff}^{+}(S,P)$, also containing $\mathrm{Diff}^{\circ}(S,P)$ as a normal subgroup.
Then the \emph{mapping class group} of $(S,P)$ is defined as the quotient 
$\mathrm{Mod}(S,P):=\mathrm{Diff}^{+}(S,P)/\mathrm{Diff}^{\circ}(S,P)$, endowed with the discrete topology.

To an embedded circle $\alpha\subset S\backslash P$ one can associate an element of $\mathrm{Mod}(S,P)$, called \emph{Dehn twist} 
and denoted by $D_{\alpha}$, as follows. Embedding the circle $\mathbb{S}^{1}$ into $\mathbb{C}$ 
such that $|s|=1$ for any $s\in \mathbb{S}^{1}$ 
and thus an orientation on $\mathbb{S}^{1}$ is given accordingly, let 
$\phi: (-1,1)\times \mathbb{S}^{1}\rightarrow S\backslash P$ be an orientation preserving open embedding such that 
$\phi(0,\mathbb{S}^{1})=\alpha$, let $\theta: (-1,1)\rightarrow [0,2\pi]$ be a smooth monotone function taking values constant $0$ 
on $(-1,-\frac{1}{2})$ and constant $2\pi$ on $(\frac{1}{2},1)$. Define a mapping $h:S\rightarrow S$ such that 
$h(\phi(t,u))=\phi(t,ue^{-\sqrt{-1}\theta(t)})$ and identity on $S\backslash \phi((-1,1)\times \mathbb{S}^{1})$. 
It is obvious that $h$ is a differmorphism and one can check its equivalence class in $\mathrm{Mod}(S,P)$ is only 
determined by the isotopy class of $\alpha$ relative to $P$. 

Such an embedded circle $\alpha$ is called a \emph{nonseparating curve} if its complement $S\backslash\alpha$ is connected. 
Then the complement becomes the interior of a compact connected surface of genus $g-1$ with two boundary components. 
Otherwise it is called a \emph{separating curve}. This circle divides $S$ into two connected components, written as $S'$ and $S''$, 
each of which is the interior of a compact surface with $\alpha$ as its boundary. If the genera of $S'$ and $S''$ are $g'$ and $g''$ 
respectively, then we have $g=g'+g''$. And the set $P$ of marked points is also divided into two subsets: 
$P':=P\cap S'$ and $P'':=P\cap S''$, lying on $S'$ and $S''$ respectively.

The following theorem is obtained by Dehn and Lickorish.

\begin{theorem}[Dehn-Lickorish]
Let $S$ and $P$ be given as above. Then the mapping class group $\mathrm{Mod}(S,P)$ of $(S,P)$ is generated by 
finitely many Dehn twists.
\end{theorem}

Recall that a conformal structure on a 2-dimensional real vector space is an inner product on it given up to a scalar 
multiplication. Then the spaces $\mathrm{Conf}(T_{p}S)$, for each $p\in S$ the set of conformal structures on the tangent space $T_{p}S$, 
form a disk bundle over $S$, denoted by $\mathrm{Conf}(S)$. A smooth section of this bundle is called a \emph{conformal structure} 
on $S$.

From now on we will always assume $P$ is finite or empty. The group $\mathrm{Diff}^{+}(S,P)$ acts on $\mathrm{Conf}(S)$. Then 
the \emph{Teichm\"uller space} of $(S,P)$ is defined as the space of conformal structures on $S$ given up to isotopy relative 
to $P$:
\[ T(S,P):=\mathrm{Diff}^{\circ}(S,P)\backslash\mathrm{Conf}(S), \]
which can also be written as $T_{g,P}$ since now the genus is enough to characterize the underlying surface. 
Likewise, we can also define another orbit space, the space of conformal structures on $S$ given up to orientation-preserving 
differmorphisms relative to $P$:
\begin{align*}
M_{g,P}=M(S,P):&=\mathrm{Diff}^{+}(S,P)\backslash\mathrm{Conf}(S) \\
&=\mathrm{Mod}(S,P)\backslash T(S,P).
\end{align*}
In fact, this is the \emph{moduli space} of Riemann surfaces of genus $g$ with $P$-marked. 
If the elements of $P$ are ordered: $P=\{p_{1},p_{2},\cdots,
p_{n}\}$, then we can write $T_{g,P}$ and $M_{g,P}$ as $T_{g,n}$ and $M_{g,n}$ respectively.

A \emph{pair of pants} is a compact surface with boundary of genus zero such that its boundary has three connected components. 
If it is endowed with a hyperbolic structure for which the three boundary components are geodesics, then it is called a 
\emph{hyperbolic pair of pants}. Since it is homeomorphic to a sphere with three holes, its Euler characteristic is 
$2-2g-n=2-0-3=-1$.

Let $S^{\circ}:=S\backslash P$. From now on we will assume $2-2g-n<0$ throughout, or equivalently, the Euler characteristic of 
$S^{\circ}$ is negative. This means if a conformal structure $J$ is given on $S$, then its restriction to $S^{\circ}$ 
underlies a complete hyperbolic metric, making any element of $P$ to be a cusp. Therefore, $S^{\circ}$ can be identified 
with a quotient of the upper half plane $\mathbb{H}$, i.e., 
$\Gamma\backslash\mathbb{H}$ with some $\Gamma\subset\mathrm{PSL}(2,\mathbb{R})$.

Then a pants decomposition of $S^{\circ}$ is given as follows.
\begin{definition}
	A \emph{pants decomposition} of $S^{\circ}$ is a closed one dimensional submanifold $A\subset S$ which does not meet $P$ 
	given up to isotopy relative to $P$ (so every connected component of $A$ is an embedded circle) such that every connected 
	component of $S^{\circ}\backslash A$ is differmorphic to the interior of a pair of pants.
\end{definition}
Now the boundary of $S^{\circ}\backslash A$ is either a connected component of $A$ or a singleton in $P$. One can then easily 
encode the topological type of the triple $(S,P;A)$ by a connected graph $G(S,P;A)$: the vertex set 
$V(S,P;A)=\pi_{0}(S^{\circ}\backslash A)$, the edge set 
$Edge(S,P;A)=\pi_{0}(A)\cup P$ with the obvious incidence relation. Those edges indexed by $\pi_{0}(A)$ are called \emph{interior edges}: 
the separating curve defines an edge connecting two vertices, while the nonseparating curve defines a loop at a vertex. 
Likewise, those edges indexed by $P$ are called the \emph{exterior edges}: they look like `leaves' which emanates from one vertex.
\begin{proposition}\label{prop:trivalent-graph}
	The graph $G(S,P;A)$ is trivalent (i.e., any vertex is of degree $3$), has $2g-2+n$ vertices and $3g-3+n$ interior edges.
\end{proposition}

\begin{proof}
Each vertex of the graph $G(S,P;A)$ is indexed by the interior of a pair of pants, which is homeomorphic to a thrice 
punctured sphere. Since it has three connected boundary components, it is clear that the graph is trivalent. 

We already know that the Euler characteristic of $S^{\circ}$, a pair of pants and a circle is $2-2g-n$, $-1$ and $0$ 
respectively. So by the additive property of Euler characteristic we can deduce that there are $2g-2+n$ connected components 
of $S^{\circ}\backslash A$ if $A$ induces a pants decomposition of $S^{\circ}$, i.e., there are $2g-2+n$ vertices in $G(S,P;A)$. 
Since the graph is trivalent, there are $3(2g-2+n)=6g-6+3n$ half edges where the interior edges are counted twice and the exterior 
edges are counted once. The number of the exterior edges is just $n$, i.e., the number of the elements of $P$. So there are 
$(6g-6+3n-n)/2=3g-3+n$ interior edges in $G(S,P;A)$.
\end{proof}

\begin{remark}
It should be noted that the pants decomposition for a given $S^{\circ}$ is not unique, which can be seen from the following picture.
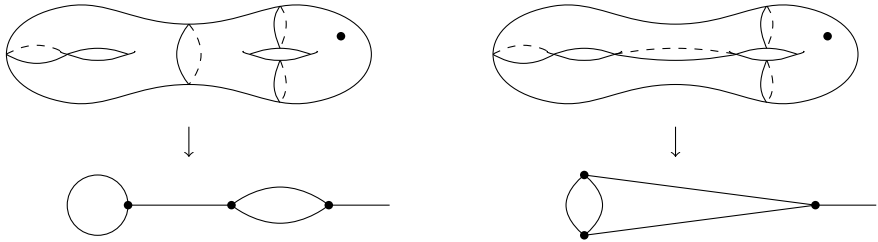
\begin{figure}[h]
\centering
\begin{tikzpicture}[scale=0.8]
\draw  plot [smooth cycle, tension=0.8] coordinates {(0,0) (1,0.8) (3,0.5) (5,0.8) (6,0) (5,-0.8) (3,-0.5) (1,-0.8)};
\draw  plot [smooth, tension=1] coordinates {(1,0) (1.5,0.1) (2,0)};
\draw  plot [smooth, tension=1] coordinates {(0.9,0.05) (1,0) (1.5,-0.1) (2,0) (2.1,0.05)};
\draw  plot [smooth, tension=1] coordinates {(4,0) (4.5,0.1) (5,0)};
\draw  plot [smooth, tension=1] coordinates {(3.9,0.05) (4,0) (4.5,-0.1) (5,0) (5.1,0.05)};
\fill (5.5,0.3) circle[radius=2pt];
\draw  [dashed] plot [smooth, tension=1] coordinates {(0,0) (0.5,0.15) (1,0)};
\draw  plot [smooth, tension=1] coordinates {(0,0) (0.5,-0.15) (1,0)};
\draw  [dashed] plot [smooth, tension=1] coordinates {(3,0.5) (3.2,0) (3,-0.5)};
\draw  plot [smooth, tension=1] coordinates {(3,0.5) (2.8,0) (3,-0.5)};
\draw  [dashed] plot [smooth, tension=1] coordinates {(4.5,0.8) (4.6,0.5) (4.5,0.1)};
\draw  plot [smooth, tension=1] coordinates {(4.5,0.8) (4.4,0.5) (4.5,0.1)};
\draw  [dashed] plot [smooth, tension=1] coordinates {(4.5,-0.1) (4.6,-0.5) (4.5,-0.8)};
\draw  plot [smooth, tension=1] coordinates {(4.5,-0.1) (4.4,-0.5) (4.5,-0.8)};

\draw  plot [smooth cycle, tension=0.8] coordinates {(8,0) (9,0.8) (11,0.5) (13,0.8) (14,0) (13,-0.8) (11,-0.5) (9,-0.8)};
\draw  plot [smooth, tension=1] coordinates {(9,0) (9.5,0.1) (10,0)};
\draw  plot [smooth, tension=1] coordinates {(8.9,0.05) (9,0) (9.5,-0.1) (10,0) (10.1,0.05)};
\draw  plot [smooth, tension=1] coordinates {(12,0) (12.5,0.1) (13,0)};
\draw  plot [smooth, tension=1] coordinates {(11.9,0.05) (12,0) (12.5,-0.1) (13,0) (13.1,0.05)};
\fill (13.5,0.3) circle[radius=2pt];
\draw  [dashed] plot [smooth, tension=1] coordinates {(8,0) (8.5,0.15) (9,0)};
\draw  plot [smooth, tension=1] coordinates {(8,0) (8.5,-0.15) (9,0)};
\draw  [dashed] plot [smooth, tension=1] coordinates {(10,0) (11,0.1) (12,0)};
\draw  plot [smooth, tension=1] coordinates {(10,0) (11,-0.1) (12,0)};
\draw  [dashed] plot [smooth, tension=1] coordinates {(12.5,0.8) (12.6,0.5) (12.5,0.1)};
\draw  plot [smooth, tension=1] coordinates {(12.5,0.8) (12.4,0.5) (12.5,0.1)};
\draw  [dashed] plot [smooth, tension=1] coordinates {(12.5,-0.1) (12.6,-0.5) (12.5,-0.8)};
\draw  plot [smooth, tension=1] coordinates {(12.5,-0.1) (12.4,-0.5) (12.5,-0.8)};

\fill (2,-2.5) circle[radius=2pt];
\fill (3.7,-2.5) circle[radius=2pt];
\fill (5.3,-2.5) circle[radius=2pt];
\draw  plot [smooth cycle, tension=1] coordinates {(2,-2.5) (1.5,-2) (1,-2.5) (1.5,-3)};
\draw  plot [smooth, tension=1] coordinates {(2,-2.5) (3.7,-2.5)};
\draw  plot [smooth, tension=1] coordinates {(3.7,-2.5) (4.5,-2.2) (5.3,-2.5)};
\draw  plot [smooth, tension=1] coordinates {(3.7,-2.5) (4.5,-2.8) (5.3,-2.5)};
\draw  plot [smooth, tension=1] coordinates {(5.3,-2.5) (6.3,-2.5)};

\fill (9.5,-2) circle[radius=2pt];
\fill (9.5,-3) circle[radius=2pt];
\fill (13.3,-2.5) circle[radius=2pt];
\draw  plot [smooth, tension=1] coordinates {(9.5,-2) (9.2,-2.5) (9.5,-3)};
\draw  plot [smooth, tension=1] coordinates {(9.5,-2) (9.8,-2.5) (9.5,-3)};
\draw  plot [smooth, tension=1] coordinates {(9.5,-2) (13.3,-2.5)};
\draw  plot [smooth, tension=1] coordinates {(9.5,-3) (13.3,-2.5)};
\draw  plot [smooth, tension=1] coordinates {(13.3,-2.5) (14.3,-2.5)};

\draw [->] (3,-1.2) to (3,-1.7);
\draw [->] (11,-1.2) to (11,-1.7);

\end{tikzpicture}
\caption{Two pants decompositions for a genus $2$ surface with $1$ marked point}
\end{figure}
\end{remark}

Let us assume that a conformal structure $J$ is given on $S^{\circ}$ and a pants decomposition induced by $A$ is fixed.
As discussed before, it underlies a natural hyperbolic structure on $S^{\circ}$. 
In fact, there is a unique closed geodesic representing each connected component $\alpha$ of $A$, called a 
\emph{geodesic representative}. And these representative geodesics are disjoint. Therefore, a length is defined by the 
conformal structure $J$ on the geodesic representative of $\alpha$, denoted by $\ell_{\alpha}(J)$. Since the length is 
invariant up to the isotopy class of the conformal structure we have a well-defined function as follows
\[
\ell_{\alpha}: T(S,P)\rightarrow\mathbb{R}_{>0}; \; [J]\mapsto \ell_{\alpha}(J).
\]
We know from hyperbolic geometry that a hyperbolic pair of pants up to isometry is determined by the lengths of its 
three boundary components. So if a pants decomposition for $S^{\circ}$ is fixed and the lengths of those geodesic representatives 
are given, then each hyperbolic pair of pants is determined accordingly up to isometry. However, from these data one still can not 
recover $(S,J)$. That is because even though we know that which pair of boundary components should be welded onto each other 
from the decomposition, there still exists a rotation over an angle for two $\mathbb{S}^{1}$ to be welded onto each other. 
This causes a so-called \emph{geodesic shearing} action along the geodesic representative of the connected component $\alpha$ 
of $A$ so that two conformal structures may differ by a geodesic shearing along the geodesic representative. Therefore, 
the above discussion 
essentially shows that $T(S,P)$ locally looks like $\mathbb{R}^{6g-6+2n}$ as a topological space 
(length function resp. geodesic shearing provides a coordinate for each 
geodesic representative and there are $3g-3+n$ interior edges for the decomposition).

Now we are ready to present the so-called \emph{Fenchel-Nielson parametrization} for the Teichm\"uller space $T(S,P)$. 
\begin{theorem}[Fenchel-Nielson parametrization]
The Fenchel-Nielson parametrization, given as above by the length function resp. geodesic shearing for each 
geodesic representative, defines a global homeomorphism from $(\mathbb{R}_{>0}\times\mathbb{R})^{\pi_{0}(A)}$ to 
$T(S,P)$. In fact, the action of geodesic shears $\mathbb{R}^{\pi_{0}(A)}$ gives rise to a 
principal fibration for $T(S,P)$ which can be illustrated by $\ell:T(S,P)\rightarrow \mathbb{R}_{>0}^{\pi_{0}(A)}$.

Moreover, the Dehn twists around the connected components of $A$ generate a free abelian subgroup of $\mathrm{Mod}(S,P)$ of 
rank $|\pi_{0}(A)|$.
\end{theorem}

\begin{remark}
We note that the Fenchel-Nielson parametrization a prior depends on the pants decomposition. But the Fenchel-Nielson parametrizations 
associated to different pants decompositions are differentiable in terms of each other. So the (smooth) manifold 
structure is independent of the 
pants decomposition. In particular, $T(S,P)$ even has a complex manifold structure.
\end{remark}

Now for each smooth Riemann surface supported on $S$ with $P$-marked, we associate to each interior edge of $G(S,P;A)$ 
the length $\ell_{\alpha}(J)$ of its corresponding geodesic representative, to each exterior edge of $G(S,P;A)$ a 
length $\infty$. By this operation the graph $G$ is made to be a metrized graph, denoted by $(G,\ell)$. 
We say that the metric graph $(G,\ell)$ 
gives a \emph{tropicalization} of the smooth Riemann surface with $P$-marked. 
Sometimes we call the graph $G(S,P;A)$ a 
\emph{topological tropicalization} of $S$ if we do not want to consider its metric structure.

Conversely, it is clear that for a given interior edge of $G$, i.e., corresponding to a geodesic representative, 
any length taking value in 
$\mathbb{R}_{>0}$ can be realized by a Riemann surface supported on $(S,P)$ leaving other edge-lengths fixed. This means 
that the graph $G$ endowed with an arbitrary length function $\ell$ on its interior edges can always 
be tropicalized from a Riemann surface. 

\begin{remark}
It is not surprising that the process of tropicalization discards some information, in this case, on the geodesic shears.
\end{remark}
	
\section{Moduli of tropicalizations}\label{sec:moduli}

In this section we first construct the moduli space of the tropicalizations of smooth $n$-pointed genus $g$ 
Riemann surfaces, 
as a topological space. And then we compactify it by adding strata which can be interpreted as parametrizing 
the tropicalizations of nodal $n$-pointed Riemann surfaces. 
We then show that the compactified space is also Hausdorff.

\subsection{The moduli space of tropicalizations}

Now let us fix a pants decomposition $A$ for $S^{\circ}$. We are thus given the tropicalization, $(G,\ell)$, of a smooth $n$-pointed 
Riemann surface supported on $S$. 
From now on we will assume that the elements in $P$ are ordered. 

The genus of a graph is defined as its first Betti number
\[
g(G):=b_{1}(G)=\dim_{\mathbb{Z}}H_{1}(G,\mathbb{Z})=|E(G)|-|V(G)|+|\{\text{connected components}\}|
\]
where $E(G)$ does not count leaves. This suggests us to distinguish the interior edge from the exterior edge, so we will 
call an interior edge just an edge, an exterior edge a \emph{leaf} in what follows. The sets of interior edges and 
exterior edges are denoted by $E(G)$ and $L(G)$ respectively.

Then we can show that the genus of $G$ is just $g$, the genus of the smooth Riemann surface from which it is tropicalized.
\begin{proposition}
The genus of the tropicalization $(G,\ell)$ of the smooth $n$-pointed Riemann surface supported on $S$ is just $g$.
\end{proposition}

\begin{proof}
We directly apply the above formula.
\begin{align*}
g(G)&=|E(G)|-|V(G)|+|\{\text{connected components}\}| \\
&=(3g-3+n)-(2g-2+n)+1 \\
&=g.
\end{align*}
The number of vertices and edges of $G$ are already obtained in Proposition $\ref{prop:trivalent-graph}$, and 
the number of connected components is just $1$ since $G$ is connected.
\end{proof}

Conversely, if we are given a trivalent genus $g$ connected graph with $n$ leaves, we can recover a genus $g$ surface with $n$ 
marked points as follows. We first associate to each vertex a pair of pants and to a half edge or a leaf which is incident 
to the vertex a boundary component, then we weld the corresponding boundary components if there is an edge 
connecting two vertices. In the meantime those boundary components corresponding to leaves are shrunk to a marked point. A topological 
surface is thus recovered. That the surface is still of genus 
$g$ is due to the above proposition because the given graph is a tropicalization of the recovered surface. It is clear that 
the surface has $n$ marked points since they are recovered from the $n$ leaves.

In order to construct the moduli space of those tropicalizations of $n$-pointed smooth Riemann surfaces supported on $S$, denoted by 
$M_{G}$, it is not surprising to introduce the open cone 
\[
\mathcal{C}(G):=\mathbb{R}_{>0}^{|E(G)|}=\mathbb{R}_{>0}^{3g-3+n}
\]
endowed with the usual topology. To every point $(\ell_{1},\dots,\ell_{|E(G)|})$ in the cone there corresponds a 
unique metric graph $(G,\ell)$ for which the length of its $i$-th edge is $\ell_{i}$.

But we notice that the symmetry of the graph would cause that some distinct graphs, in the sense of corresponding to 
distinct points in the cone, 
are isomorphic to each other. 
So before we can go further, we need to introduce the automorphism group of a graph first. For that we define 
a so-called \emph{endpoint} map
\[
\epsilon: E(G)\cup L(G)\rightarrow V(G)
\]
which assigns an edge (or a leaf) its endpoints. The image of a non-loop edge under this map thus should be 
understood as a value of unordered two vertices.

\begin{definition}
A map $f:V(G)\cup E(G)\cup L(G)\rightarrow V(G')\cup E(G')\cup L(G')$ is called a morphism from graph $G$ to graph $G'$ 
if we have $f(L(G))\subset L(G')$, and the diagram below is commutative.
\[
\begin{tikzpicture}[scale=2]
\node (A) at (0,1) {$V(G)\cup E(G)\cup L(G)$};
\node (B) at (3,1) {$V(G')\cup E(G')\cup L(G')$};
\node (C) at (0,0) {$V(G)\cup E(G)\cup L(G)$};
\node (D) at (3,0) {$V(G')\cup E(G')\cup L(G')$};
\draw
(A) edge[->,>=angle 90]   node[above] {$f$}    (B)
(A) edge[->,>=angle 90]   node[right] {$(id_{V},\epsilon)$}      (C)
(B) edge[->,>=angle 90]   node[right] {$(id_{V'},\epsilon')$}     (D)
(C) edge[->,>=angle 90]   node[above] {$f$}     (D);
\end{tikzpicture}
\]

The morphism $f$ is an isomorphism if it induces, by restriction, three bijections 
$f_{V}:V(G)\rightarrow V(G')$, $f_{E}:E(G)\rightarrow E(G')$ and $f_{L}:L(G)\rightarrow L(G')$. 
Then an automorphism of $G$ is an isomorphism between $G$ and itself.
\end{definition}


Besides the notion of an automorphism of a graph, there is also a common type of graph morphism, called \emph{(edge) contraction}, 
playing an important role in our paper. Let $Q\subset E(G)$ be a set of edges, we denote by $G_{/Q}$ the graph removing all the 
edges in $Q$ and identifying the endpoints of each edge in $Q$. Then we have a resulting map $f:G\rightarrow G_{/Q}$ 
by fixing anything outside of $Q$, called \emph{(edge) contraction}. On the other hand, 
let $E_{1}:=E(G)\backslash Q$, we can also obtain a graph 
$G-E_{1}$ by just removing all the edges of $E_{1}$. Then every connected component of $G-E_{1}$ can be contracted to a vertex of 
$G_{/Q}$; conversely, the preimage $f^{-1}(v')\subset G$ for every vertex $v'\in V(G_{/Q})$ is a connected component of 
$G-E_{1}$. In particular, we have 
\begin{align}\label{eqn:contraction-of-b1}
b_{1}(G-E_{1})=\sum_{v'\in V(G_{/Q})}b_{1}(f^{-1}(v')).
\end{align}
By the additivity of $b_{1}$ we have that 
\begin{align}\label{eqn:additivity-of-b1}
b_{1}(G)=b_{1}(G_{/Q})+b_{1}(G-E_{1}).
\end{align}

It is clear that all the automorphisms of a graph $G$ form a group, denoted by $\mathrm{Aut}(G)$. By passing from 
$f$ to $f_{E}$ it induces a homomorphism from $\mathrm{Aut}(G)$ to the symmetry group on $|E(G)|$ edges. Hence $\mathrm{Aut}(G)$ 
acts on the open cone $\mathbb{R}_{>0}^{|E|}$ by permuting the coordinates. Therefore, we construct the moduli space 
of tropicalizations $(G,\ell)$ with $G$ underlying the given pants decomposition:
\[
M_{G}:=\mathcal{C}(G)/\mathrm{Aut}(G)=\mathbb{R}_{>0}^{|E(G)|}/\mathrm{Aut}(G)=\mathbb{R}_{>0}^{3g-3+n}/\mathrm{Aut}(G)
\]
with the quotient topology.

\begin{remark}
By passing from $f$ to $f_{E}$ one may lose some non-trivial elements while fixing the edge set $E(G)$, for instance, 
an element which is not identity on $V(G)$ but fixing the edge set. 
That means $\mathrm{Aut}(G)$ may contain non-trivial elements acting 
trivially on $\mathcal{C}(G)$.
\end{remark}

\subsection{Weighted contractions}

Now we want to add the boundary strata to the open cone $\mathcal{C}(G)$ in order to get the closed cone 
$\mathcal{C}(G)^{+}=\mathbb{R}_{\geq 0}^{|E(G)|}$. That means some edges of the metric graph $(G,\ell)$ would go to zero, 
leading to the contraction of $G$. 
Since the length of each edge of the graph is defined as the length of its corresponding geodesic representative of $\alpha\in A$, 
the length going to zero amounts to shrinking the corresponding geodesic representative until we get a node. By this way we get a 
nodal Riemann surface, and we say that the contracted graph $G'$ is a tropicalization of this nodal Riemann surface.

\begin{remark}
Shrinking a closed curve $\alpha$ in a smooth Riemann surface leads to a nodal Riemann surface. 
But the tropicalization of the nodal Riemann surface a prior depends on the tropicalization $G$ of the 
smooth one, hence it is not unique as well. 
On the other hand, if we start from a tropicalization $G$ for a smooth one, unfortunately 
the various contractions of $G$ 
would not exhaust all the nodal Riemann surfaces of arithmetic genus $g$. 
This point of view will become more clear in the next section when 
we compare this moduli space with the moduli space of Riemann surfaces.
\end{remark}

We know that those nodal Riemann surfaces, obtained by shrinking some $\alpha$'s from a smooth one, still have 
arithmetic genus $g$. For a graph, however, after contracting an edge, its genus might decrease. So we introduce 
a weighted version for these notions, provided by Brannetti-Melo-Viviani \cite{Brannetti-Melo-Viviani}, 
so as to remedy this problem.

\begin{definition}
A graph $G$ is called a \emph{weighted graph} if it is endowed with a weight function on 
the vertices $w:V(G)\rightarrow \mathbb{Z}_{\geq 0}$, 
denoted by $(G,w)$. We can think of an unweighted graph as a $\underline{0}$-weighted graph $(G,\underline{0})$.

A \emph{weighted metric graph} $(G,w,\ell)$ is defined likewise if we start from a metric graph $(G,\ell)$.
\end{definition}

Then the genus of a weighted graph is modified by adding weights of vertices to the first Betti number:
\[
g(G,w):=b_{1}(G)+\sum_{v\in V(G)}w(v).
\] 
Therefore, the contraction of a weighted graph can be naturally modified by setting the weight of each vertex of the contracted 
graph as the genus of its weighted preimage.

\begin{definition}\label{def:weight-of-vertex}
Let $Q\subset E(G)$ be a set of edges of a weighted graph $(G,w)$. A contraction of $Q$ in $(G,w)$ is called a 
\emph{weighted contraction} if the weight function $w_{/Q}$ on the contracted graph $G_{/Q}$ is defined as 
follows:
\[
w_{/Q}(v'):=b_{1}(f^{-1}(v'))+\sum_{v\in f_{V}^{-1}(v')}w(v)
\]
for each vertex $v'\in V(G_{/Q})$.
\end{definition}
We immediately notice that the weighted contraction would not change the genus of a weighted graph. 

\begin{proposition}
Using the notations above, we have 
\[
g(G_{/Q},w_{/Q})=g(G,w).
\]
\end{proposition}

\begin{proof}
Let $Q$	and $E_{1}$ be defined as above. Then we have 
\begin{align*}
g(G,w)&=b_{1}(G)+\sum_{v\in V(G)}w(v)  \\
&=b_{1}(G_{/Q})+b_{1}(G-E_{1})+\sum_{v\in V(G)}w(v)  \hskip 2.2cm \text{by (\ref{eqn:additivity-of-b1})} \\
&=b_{1}(G_{/Q})+\sum_{v'\in V(G_{/Q})}b_{1}(f^{-1}(v'))+\sum_{v\in V(G)}w(v)  \hskip 0.7cm \text{by (\ref{eqn:contraction-of-b1})}  \\
&=b_{1}(G_{/Q})+\sum_{v'\in V(G_{/Q})}b_{1}(f^{-1}(v'))+\sum_{v'\in V(G_{/Q})}\sum_{v\in f_{V}^{-1}(v')}w(v)  \\ 
&=b_{1}(G_{/Q})+\sum_{v'\in V(G_{/Q})}\Big(b_{1}(f^{-1}(v'))+\sum_{v\in f_{V}^{-1}(v')}w(v)\Big)  \\ 
&=b_{1}(G_{/Q})+\sum_{v'\in V(G_{/Q})}w_{/Q}(v') \hskip 2cm \text{by Definition \ref{def:weight-of-vertex}}  \\ 
&=g(G_{/Q},w_{/Q}).
\end{align*}
\end{proof}

If we regard the tropicalization $G$ for our smooth Riemann surfaces as a $\underline{0}$-weighted graph, then by this 
proposition we can see that all the weighted contractions of this weighted graph $(G,\underline{0})$ still stay in the 
category of genus $g$ weighted graph (with $n$ leaves).  And we denote that  
$(G',w')\preceq (G,w)$ if $(G',w')$ is a weighted contraction of $(G,w)$. In what follows we will call such a $(G,\underline{0})$, 
tropicalized from a smooth Riemann surface, a \emph{regular tropicalization}. It is clear that a regular tropicalization 
is always trivalent and $\underline{0}$-weighted.

We can define the automorphism group $\mathrm{Aut}(G,w)$ of a weighted graph $(G,w)$ as the subgroup of $\mathrm{Aut}(G)$ 
preserving the weights on vertices. Then if a tropicalization of a (singular) Riemann surface is $(G,w)$, we can also define 
the moduli space of tropicalizations based on this weighted graph $(G,w)$ as follows:
\[
M^{\mathrm{tr}}(G,w):=\mathcal{C}(G,w)/\mathrm{Aut}(G,w)=\mathbb{R}_{>0}^{|E(G)|}/\mathrm{Aut}(G,w).
\]
Hence $M_{G}$ can be rewritten as $M^{\mathrm{tr}}(G,\underline{0})$, and $\mathrm{Aut}(G)$ as $\mathrm{Aut}(G,\underline{0})$.
Then the space $\mathcal{C}(G,w)^{+}$ can be decomposed as follows 
\begin{align*}
\mathcal{C}(G,w)^{+}=\bigsqcup_{Q\subset E(G)}\mathcal{C}(G_{/Q},w_{/Q}).
\end{align*}

Now let us see how the action of the group $\mathrm{Aut}(G,w)$ extends to the closed cone $\mathcal{C}(G,w)^{+}$, in order to 
define the partially compactified moduli space $M^{\mathrm{tr}}(G,w)^{+}$. For an open 
face of $\mathcal{C}(G,w)^{+}$, i.e., the open cone of some weighted contraction $(G',w')$ of $(G,w)$, there is a natural 
action of $\mathrm{Aut}(G',w')$ so that we can similarly define $M^{\mathrm{tr}}(G',w')=\mathbb{R}_{>0}^{|E(G')|}/\mathrm{Aut}(G',w')$. 
But it may also happen that there are two distinct contractions from $(G,w)$ being isomorphic, i.e., there is an isomorphism 
$f:(G',w')\xrightarrow{\backsimeq} (G'',w'')$. Then we have an isometry 
$\phi_{f}:\mathcal{C}(G',w')\xrightarrow{\backsimeq} \mathcal{C}(G'',w'')$ induced by $f$. We also want to identify these 
isomorphic tropicalizations since the surfaces supported on these tropicalizations are of the same topological type. 
Therefore, we extend the action of $\mathrm{Aut}(G,w)$ to the closed cone $\mathcal{C}(G,w)^{+}$ in such a manner that 
the quotient space 
\[
M^{\mathrm{tr}}(G,w)^{+}:=\mathcal{C}(G,w)^{+}/\mathrm{Aut}(G,w)=\mathbb{R}_{\geq 0}^{|E(G)|}/\mathrm{Aut}(G,w)
\]
identifies isomorphic tropicalizations. And this quotient space endowed with the quotient topology 
is our desired partially compactified moduli space for $M^{\mathrm{tr}}(G,w)$.

\begin{lemma}\label{lem:finite-fibers}
The quotient map $\pi^{+}:\mathcal{C}(G,w)^{+}\rightarrow M^{\mathrm{tr}}(G,w)^{+}$ has finite fibers.
\end{lemma}

\begin{proof}
Since our $(G,w)$ is a finite graph, its automorphism group $\mathrm{Aut}(G,w)$ is finite. So is any weighted contraction of it. 
On the other hand it has only finitely many weighted contractions, so for a given weighted contraction, there can only be finitely many 
other weighted contractions being isomorphic to it. Therefore the preimage of a point in $M^{\mathrm{tr}}(G,w)^{+}$ 
contains only finitely many points in $\mathcal{C}(G,w)^{+}$.
\end{proof}

\begin{lemma}\label{lem:strata}
$M^{\mathrm{tr}}(G',w')\subset M^{\mathrm{tr}}(G,w)^{+}$ if and only if $(G',w')\preceq (G,w)$.
\end{lemma}

\begin{proof}
This is a direct consequence of the fact that $\mathcal{C}(G',w')\subset \mathcal{C}(G,w)^{+}$ 
if and only if $(G',w')\preceq (G,w)$.
\end{proof}

\begin{remark}\label{rem:two-types}
By the argument in Lemma $\ref{lem:finite-fibers}$ as well as the preceding discussion, 
we know that the extension of the action of $\mathrm{Aut}(G,w)$ to the boundary of 
$\mathcal{C}(G,w)$ consists of two types: 1) the action of $\mathrm{Aut}(G_{/Q},w_{/Q})$ on the open face 
$\mathcal{C}(G_{/Q},w_{/Q})$, and 2) the isometry $\phi_{f}$ between faces induced by the isomorphism $f$ 
between weighted contractions.
\end{remark}

\subsection{Extended tropicalizations}
Nevertheless, one can notice that the space $M^{\mathrm{tr}}(G,w)^{+}$ is still not compact. This can easily be seen by allowing 
the edge-lengths to go to arbitrarily large. In order to solve this problem, we introduce the so-called extended length function 
for edges by allowing their lengths to be infinity. Namely, now we have an ``extended" length function 
\[\ell:E(G)\rightarrow \mathbb{R}_{\geq 0}\cup \{\infty\},
\]
then the tropicalization $(G,w,\ell)$ is called an \emph{extended tropicalization}.

In fact, letting the length go to infinity amounts to enlarging the corresponding geodesic representative until the Riemann 
surface is cut open from there, and each of the two boundary components becomes a marked point. In other words, 
the edge can be understood 
as being broken up into two leaves. By this operation, depending on the geodesic representative is a separating curve or not, 
we get 1) either a new Riemann surfaces with two connected components 
for which the sum of their genera is $g$ and there is one new marked point on each component, 2) 
or a new Riemann surface with genus being $1$ less and having two new marked points on it. Then we say that the 
graph $(G,w,\ell)$ is a tropicalization of this (possibly disconnected) Riemann surface. 

The Riemann surfaces in these two cases are still 
smooth but there is also some defect in each case: for the first case, the Riemann surface is not connected any more; 
while in the latter case the Riemann surface is no longer of genus $g$. So these are not the Riemann surfaces that we 
want to tropicalize for our moduli space.

We immediately, however, notice that the Riemann surface corresponding to edge-length going to infinity is just the 
normalization of the nodal Riemann surface when the length of the same edge goes to zero. This fact happens to connect 
the two tropicalizations when edge goes to $\infty$ and $0$ respectively: 
the Riemann surface from which the former tropicalization comes is 
just the normalization of the latter one.

It suggests that we can identify $0$ and $\infty$ for each edge so that $(\mathbb{R}_{\geq 0}\cup \{\infty\})/\sim$ is 
homeomorphic to 
$\mathbb{S}^{1}$ and hence is compact. The explicit homeomorphism between them can be given by 
\[
h:(\mathbb{R}_{\geq 0}\cup \{\infty\})/\sim\rightarrow \mathbb{S}^{1}; \quad x\mapsto \exp(2\pi\sqrt{-1}\frac{x}{x+1}),
\]
where $0\sim\infty$. And one possible distance function on $(\mathbb{R}_{\geq 0}\cup \{\infty\})/\sim$ compatible 
with this topology can be given by 
\[
d(x,y):=2\pi\min(|\frac{x}{x+1}-\frac{y}{y+1}|,1-|\frac{x}{x+1}-\frac{y}{y+1}|)
\]
when a metric structure is needed.

By taking the homeomorphism $h$ on each coordinate we can have a homeomorphism from 
$((\mathbb{R}_{\geq 0}\cup \{\infty\})/\sim)^{|E(G)|}$ 
to $(\mathbb{S}^{1})^{|E(G)|}$, denoted by $h$ as well. 
Thanks to this homeomorphism, we will not distinguish the open cone 
$\mathcal{C}(G,w)$ with its image under $h$ in the topological sense in what follows. 
Similarly, a metric structure on $\big((\mathbb{R}_{\geq 0}\cup \{\infty\})/\sim\big)^{|E(G)|}$ compatible with the topology 
is just the product of the metric on each coordinate given above.

Now we can define a compactification of the open cone $\mathcal{C}(G,w)$, denoted by 
$\overline{\mathcal{C}(G,w)}$, as follows:
\[
\overline{\mathcal{C}(G,w)}:=(\mathbb{S}^{1})^{|E(G)|}.
\]

\begin{remark}
	We can immediately notice that there is a bijection between the closed cone $\mathcal{C}(G,w)^{+}$ and 
	the space $\overline{\mathcal{C}(G,w)}$, although they are endowed with different topology.
\end{remark}

Likewise the space $\overline{\mathcal{C}(G,w)}$ can be decomposed as follows 
\begin{align}\label{eqn:decomposition}
\overline{\mathcal{C}(G,w)}=\bigsqcup_{Q\subset E(G)}\mathcal{C}(G_{/Q},w_{/Q}).
\end{align}

As same as we did for $\mathcal{C}(G,w)^{+}$, we extend the action of $\mathrm{Aut}(G,w)$ to the compactified space 
 $\overline{\mathcal{C}(G,w)}$ in such a manner that 
the quotient space 
\[
\overline{M^{\mathrm{tr}}(G,w)}:=\overline{\mathcal{C}(G,w)}/\mathrm{Aut}(G,w)=(\mathbb{S}^{1})^{|E(G)|}/\mathrm{Aut}(G,w)
\]
identifies isomorphic tropicalizations. 
This quotient space endowed with the quotient topology 
is our desired compactified moduli space for $M^{\mathrm{tr}}(G,w)$.

\begin{remark}
	We only consider weighted contractions rather than extended tropicalizations in our space $\overline{\mathcal{C}(G,w)}$, 
	although we identify $0$ and $\infty$ for each edge. That is because we want our 
	tropicalizations to come from connected $n$-pointed genus $g$ Riemann surfaces.
\end{remark}

By the same argument for Lemma $\ref{lem:finite-fibers}$ and $\ref{lem:strata}$ we readily have the following similar 
properties for $\overline{M^{\mathrm{tr}}(G,w)}$.

\begin{lemma}
\begin{enumerate}
	\item The quotient map $\overline{\pi}:\overline{\mathcal{C}(G,w)}\rightarrow \overline{M^{\mathrm{tr}}(G,w)}$ has finite fibers.

	\item $M^{\mathrm{tr}}(G',w')\subset \overline{M^{\mathrm{tr}}(G,w)}$ if and only if $(G',w')\preceq (G,w)$.
\end{enumerate}
\end{lemma}

Now we are ready to see what this space $\overline{M^{\mathrm{tr}}(G,w)}$ looks like.

\begin{proposition}\label{prop:decomposition}
	Let $(G,w)$ be some weighted contraction coming from a regular tropicalization of a smooth pointed Riemann surface. Then we have 
\begin{enumerate}
	\item There is a stratification for $\overline{M^{\mathrm{tr}}(G,w)}$ as follows
	\[
	\overline{M^{\mathrm{tr}}(G,w)}=\bigsqcup_{(G',w') \preceq (G,w) }M^{\mathrm{tr}}(G',w'),
	\]
	where $M^{\mathrm{tr}}(G,w)$ is open and dense 
	in $\overline{M^{\mathrm{tr}}(G,w)}$. 
	
	\item $\overline{M^{\mathrm{tr}}(G,w)}$ is compact and Hausdorff  as a topological space.
\end{enumerate}
\end{proposition}

\begin{proof}
The stratification for $\overline{M^{\mathrm{tr}}(G,w)}$ comes directly from the decomposition ($\ref{eqn:decomposition}$) 
for $\overline{\mathcal{C}(G,w)}$ and the definition of $\overline{M^{\mathrm{tr}}(G,w)}$.

Since $M^{\mathrm{tr}}(G,w)$ is defined as $\mathcal{C}(G,w)/\mathrm{Aut}(G,w)$ and $\mathcal{C}(G,w)$ is open and dense in 
$\overline{\mathcal{C}(G,w)}$, by $\overline{\pi}^{-1}(M^{\mathrm{tr}}(G,w))=\mathcal{C}(G,w)$ we also have 
$M^{\mathrm{tr}}(G,w)$ is open and dense in $\overline{M^{\mathrm{tr}}(G,w)}$. This completes the proof of the statement (1).

Since $\overline{\mathcal{C}(G,w)}$ is compact, it is clear that $\overline{M^{\mathrm{tr}}(G,w)}$ is compact as well by 
definition. 

Now we show the Hausdorffness of $\overline{M^{\mathrm{tr}}(G,w)}$. For the simplicity of notations, we denote 
the open sectional face $\mathcal{C}(G_{/Q},w_{/Q})$ (they are not open faces any more in $\overline{\mathcal{C}(G,w)}$ due to different 
topology) by $F_{Q}$ for any $Q\subset E(G)$, and its closure in $\overline{\mathcal{C}(G,w)}$ by 
$\overline{F_{Q}}$. As discussed in Remark $\ref{rem:two-types}$, 
there are two types of maps leading to identifying a point in $\overline{\mathcal{C}(G,w)}$ with other points in order to get 
$\overline{M^{\mathrm{tr}}(G,w)}$: 

\noindent (i) An element $g\in \mathrm{Aut}(G_{/Q},w_{/Q})$ identifying $p$ with $gp$ for any $p\in F_{Q}$.

\noindent (ii) An isometry $\phi_{f}:F_{Q}\rightarrow F_{Q'}$, induced by an isomorphism 
$f:(G_{/Q},w_{/Q})\rightarrow (G_{/Q'},w_{/Q'})$, identifying $p$ with $\phi_{f}(p)$ for any $p\in F_{Q}$.

Let $\overline{u}$ and $\overline{v}$ be two distinct points in $\overline{M^{\mathrm{tr}}(G,w)}$. Write their preimages 
as $\overline{\pi}^{-1}(\overline{u})=\{u_{1},\dots,u_{s}\}$ and $\overline{\pi}^{-1}(\overline{v})=\{v_{1},\dots,v_{t}\}$ 
respectively. Then there exists a sufficiently small $\varepsilon>0$ (using the metric structure on $\overline{\mathcal{C}(G,w)}$) 
such that the following holds.

\noindent (a) For every $i,j$ the open balls $B(u_{i},\varepsilon)$ and $B(v_{j},\varepsilon)$ do not intersect each other.

\noindent (b) If $\overline{F_{Q}}\cap B(u_{i},\varepsilon)\neq \emptyset$ then $u_{i}\in \overline{F_{Q}}$; likewise, 
if $\overline{F_{Q}}\cap B(v_{j},\varepsilon)\neq \emptyset$ then $v_{j}\in \overline{F_{Q}}$.

Let $U:=\cup_{i=1}^{s}B(u_{i},\varepsilon)$ and $V:=\cup_{j=1}^{t}B(v_{j},\varepsilon)$. It is clear that $U$ and $V$ are 
open subsets of $\overline{\mathcal{C}(G,w)}$, and they are disjoint by (a) above.

We claim that 
\[
\overline{\pi}^{-1}(\overline{\pi}(U))=U \quad \text{and} \quad \overline{\pi}^{-1}(\overline{\pi}(V))=V.
\]

We extend the preceding action (i) and map (ii) to $\overline{\mathcal{C}(G,w)}$ by fixing anything outside of $\overline{F_{Q}}$. 
Namely, we set $gp=p$ and $\phi_{f}(p)=p$ for any $p\in\overline{\mathcal{C}(G,w)}\backslash\overline{F_{Q}}$, every 
$g\in \mathrm{Aut}(G_{/Q},w_{/Q})$ and every $\phi_{f}$ as above.

Now in order to prove the claim, it suffices to prove that the open subset $U$ is invariant under the extended action (i) 
and the extended map (ii). 

For that we pick $Q$ and let $g\in \mathrm{Aut}(G_{/Q},w_{/Q})$. As $g$ acts trivially outside $\overline{F_{Q}}$ we can 
assume that $\overline{F_{Q}}\cap U\neq\emptyset$. That means there exists a $u_{i}$ such that 
$\overline{F_{Q}}\cap B(u_{i},\varepsilon)\neq\emptyset$. This implies that $u_{i}\in \overline{F_{Q}}$ by (b) above. 
Without loss of generality, we can assume that $u_{i}\in F_{Q}$ for otherwise if $u_{i}\in \overline{F_{Q}}\backslash F_{Q}$ 
we can always find a $Q'\supset Q$ such that $u_{i}\in F_{Q'}$. Hence we have $gu_{i}\in \overline{\pi}^{-1}(\overline{u})$ 
and thus 
\[
(\overline{F_{Q}}\cap B(u_{i},\varepsilon))^{g}=\overline{F_{Q}}\cap B(gu_{i},\varepsilon)\subset U.
\]
This proves the invariance of $U$ under the extended action of (i).

To prove the invariance of $U$ under the extended map (ii). Let $p\in F_{Q}\cap U$, then there exists a $u_{i}$ such that 
$p\in \overline{F_{Q}}\cap B(u_{i},\varepsilon)$. Let $f$ be an isomorphism from $(G_{/Q},w_{/Q})$ to another contraction 
$(G_{/Q'},w_{/Q'})$ and $\phi_{f}$ be the induced isometry from $\overline{F_{Q}}$ to $\overline{F_{Q'}}$. Then we have 
\[
\phi_{f}(\overline{F_{Q}}\cap B(u_{i},\varepsilon))=\overline{F_{Q'}}\cap B(\phi_{f}(u_{i}),\varepsilon)\subset U
\]
since we know that $\phi_{f}(u_{i})\in \overline{\pi}^{-1}(\overline{u})$ by (ii) above.

Now the claim is proved. This yields that $\overline{\pi}(U)$ and $\overline{\pi}(V)$ are open and disjoint in 
$\overline{M^{\mathrm{tr}}(G,w)}$. Since $\overline{u}\in\overline{\pi}(U)$ and $\overline{v}\in\overline{\pi}(V)$, 
the Hausdorffness of $\overline{M^{\mathrm{tr}}(G,w)}$ is proved.
\end{proof}

By applying this proposition to a regular tropicalization of smooth $n$-pointed Riemann surfaces, we immediately have 

\begin{theorem}\label{thm:hausdorff}
	Let $(G,\underline{0})$ be a regular tropicalization of a smooth $n$-pointed Riemann surface. Then we have 
	\begin{enumerate}
		\item There is a stratification for $\overline{M^{\mathrm{tr}}(G,\underline{0})}$ as follows
		\[
		\overline{M^{\mathrm{tr}}(G,\underline{0})}=\bigsqcup_{(G',w') \preceq (G,\underline{0})}M^{\mathrm{tr}}(G',w'),
		\]
		where $M^{\mathrm{tr}}(G,\underline{0})$ is open and dense 
		in $\overline{M^{\mathrm{tr}}(G,\underline{0})}$. 
		
		\item $\overline{M^{\mathrm{tr}}(G,\underline{0})}$ is compact and Hausdorff as a topological space.
	\end{enumerate}
\end{theorem}

\begin{remark}
If we choose another regular tropicalization $(G',\underline{0})$ for this smooth $n$-pointed Riemann surface, we still have 
the above theorem for $\overline{M^{\mathrm{tr}}(G',\underline{0})}$. But these two moduli spaces as orbifolds are different 
because the automorphism group of the tropicalizations may not be the same.
\end{remark}

\begin{question}
Can we connect these two moduli spaces through the moduli space of Riemann surfaces or corresponding Teichm\"uller space? 
Since $(G,\underline{0})$ and $(G',\underline{0})$ are both tropicalized from a smooth $n$-pointed genus $g$ Riemann surface.
\end{question}

\section{Comparing with other moduli spaces}\label{sec:comparison}

In this section we compare our moduli space of tropicalizations with the moduli space of Riemann surfaces (or equivalently, 
complex algebraic curves), 
since this is the main goal for us to introduce this moduli space, for which we hope it could provide a new angle to 
understand the tropicalization of the moduli space of algebraic curves. 
Next we also compare it with the moduli space of tropical curves in the sense 
of \cite{Caporaso-2013}, since a tropicalization is essentially a tropical curve.  

In fact, in \cite{Caporaso-2013}, Caporaso established the so-called moduli space of $n$-pointed genue $g$ tropical curves, 
$\overline{M}_{g,n}^{\mathrm{trop}}$. And in \cite{Abramovich-Caporaso-Payne}, Abramovich, Caporaso and Payne identify this space 
with the skeleton $\overline{\Sigma}(\overline{\mathcal{M}}_{g,n})$ of the moduli stack of curves $\overline{\mathcal{M}}_{g,n}$, 
via the analytification of its coarse moduli space, $\overline{M}_{g,n}^{\mathrm{an}}$. 
This can be illustrated by the following commutative diagram.

\[
\begin{tikzpicture}[scale=2]
\node (A) at (0,1) {$\overline{M}_{g,n}^{\mathrm{an}}$};
\node (B) at (3,1) {$\overline{\Sigma}(\overline{\mathcal{M}}_{g,n})$};
\node (D) at (3,0) {$\overline{M}_{g,n}^{\mathrm{trop}}$};
\draw
(A) edge[->,>=angle 90]   node[above] {$\text{Retraction}$}    (B)
(B) edge[->,>=angle 90]   node[right] {$\backsimeq$}     (D)
(A) edge[->,>=angle 90]   node[left] {$\text{Tropicalization}$}     (D);
\end{tikzpicture}
\]

There is also a stratification on the moduli stack $\overline{\mathcal{M}}_{g,n}$ which can be described in a combinatorial way: 
\[
\overline{\mathcal{M}}_{g,n}=\bigsqcup_{\text{$(G,w)$ stable, genus $g$, $n$ leaves}}\mathcal{M}^{\mathrm{alg}}(G,w),
\]
where $\mathcal{M}^{\mathrm{alg}}(G,w)$ parametrizes those curves whose dual graph is isomorphic to $(G,w)$. Here by stable graphs 
we mean the degree of the vertex is required to be $\geq 3$ (resp. $\geq 1$) if $w(v)=0$ (resp. $w(v)=1$). The codimension of 
$\mathcal{M}^{\mathrm{alg}}(G,w)$ in $\overline{\mathcal{M}}_{g,n}$ is just the number of edges of $(G,w)$, and 
$\mathcal{M}^{\mathrm{alg}}(G',w')\subset \overline{\mathcal{M}^{\mathrm{alg}}(G,w)}$ if and only if $(G,w)\preceq (G',w')$.

Therefore, we see that there is an \emph{order-reversing} correspondence between the moduli space of algebraic curves and 
the moduli space of tropical curves.
We will, however, show that there is an \emph{order-preserving} correspondence, in a partial way though, between our moduli space of 
tropicalizations and the moduli space of Riemann surfaces.

\subsection{Comparing with the moduli space of Riemann surfaces}
Since nodal Riemann surfaces are parametrized by the boundary strata of the Deligne-Mumford space $\overline{M}_{g,n}$, 
it is necessary for us to see how to tropicalize a nodal Riemann surface before establishing the correspondence between the two 
moduli spaces.

If we are given a nodal Riemann surface in $\overline{M}_{g,n}$, 
its normalization gives rise to a (possibly disconnected) smooth Riemann surface. Every node 
marks a point on each branch on which it lies. By this operation each component becomes a pointed smooth Riemann surface 
and we tropicalize it as before. 
Thus we get a tropicalization for each smooth component and we join each pair of leaves, coming from the same node, to get an edge. 
Hence we get a connected trivalent graph of genus $g$: joining leaves would not change the degree of vertex; separating node plays no 
role for the genus; joining leaves corresponding to a nonseparating node makes the $1$ less genus to come back. 
Then we contract those newly formed edges so as to get a weighted contraction.  
We say this weighted contraction is a tropicalization of the nodal Riemann surface. Note that the tropicalization 
is not unique since the tropicalization of a smooth Riemann surface is not unique.

Now we still fix a regular tropicalization $(G,\underline{0})$ for an $n$-pointed genus $g$ Riemann surface. 
Let $(G',w')$ be a weighted contraction of $(G,\underline{0})$, then the dimension of $M^{\mathrm{tr}}(G',w')$ 
in $\overline{M^{\mathrm{tr}}(G,\underline{0})}$ is equal to $|E(G')|$.

In the meantime, we also know that $(G',w')$ underlies some nodal Riemann surface by letting those 
geodesic representatives corresponding to contracted edges go to zero. From this point of view, 
we denote by $M^{\mathrm{rs}}(G'_{\preceq G},w')$ the locus in 
$\overline{M}_{g,n}$ of those nodal Riemann surfaces by shrinking those geodesic representatives, corresponding to 
contracted edges in $(G,\underline{0})$, to a point. Here the superscript ``rs'' stands for ``Riemann surface''.

\begin{remark}
The space $M^{\mathrm{rs}}(G'_{\preceq G},w')$, unfortunately, also depends on the regular tropicalization 
$(G,\underline{0})$. The graph 
$(G',w')$ may be contracted from some other regular tropicalization, say $(G_{1},\underline{0})$. 
By shrinking corresponding geodesic representatives we may get a nodal Riemann surface of distinct topological type.

On the other hand, suppose $(G'',w'')$ is another weighted contraction of $(G,\underline{0})$ which is not isomorphic to $(G',w')$, 
the moduli space $M^{\mathrm{rs}}(G''_{\preceq G},w'')$ may be the same as $M^{\mathrm{rs}}(G'_{\preceq G},w')$, since they may 
support surfaces of the same topological type.
\end{remark} 

\begin{lemma}\label{lem:irre-and-dim}
Fix a regular tropicalization $(G,\underline{0})$ of a smooth $n$-pointed genus $g$ Riemann surface. 
Let $(G',w')$ be a weighted contraction of 
it. We have that $M^{\mathrm{rs}}(G'_{\preceq G},w')$ is an irreducible quasi-projective variety and its dimension is equal to $|E(G')|$.
\end{lemma}

\begin{proof}
Let $(C,\underline{p})\in M^{\mathrm{rs}}(G'_{\preceq G},w')$ where $\underline{p}$ is an ordered form of $P$. 
Then $d:=3g-3+n-|E(G')|$ is the number of its nodes by 
construction. Denote by $C_{1},\dots,C_{k}$ its irreducible components, and by $n_{i}$ the number of marked points on 
$C_{i}$. It is clear that $\sum_{i=1}^{k}n_{i}=n$.

We take the normalization of $C$ and write it as $\nu:\bigsqcup_{i=1}^{k}C_{i}^{\nu}\rightarrow C$, where each $C_{i}^{\nu}$ is a 
smooth genus $g_{i}$ Riemann surface with $n_{i}+d_{i}$ marked points. Here $d_{i}:=|\nu^{-1}(C_{\mathrm{sing}})\cap C_{i}^{\nu}|$.
We notice that $(C_{i}^{\nu},\underline{p}^{(i)})$ falls into $M_{g_{i},n_{i}+d_{i}}$ since the normalization preserves the 
stability of Riemann surfaces. Then it gives rise to such a following surjective morphism:
\[
M_{g_{1},n_{1}+d_{1}}\times \cdots \times M_{g_{k},n_{k}+d_{k}}\rightarrow M^{\mathrm{rs}}(G'_{\preceq G},w').
\]
This morphism can be naturally extended to the following morphism:
\[
\overline{M}_{g_{1},n_{1}+d_{1}}\times \cdots \times \overline{M}_{g_{k},n_{k}+d_{k}}\rightarrow 
\overline{M}_{g,n}.
\]
Since $M^{\mathrm{rs}}(G'_{\preceq G},w')$ is open in its image, it is quasi-projective.
It is well-known that $\overline{M}_{g_{i},n_{i}+d_{i}}$ is irreducible of dimension $3g_{i}-3+n_{i}+d_{i}$ for all 
$i=1,\dots,k$. Hence $M^{\mathrm{rs}}(G'_{\preceq G},w')$ is irreducible as well.

The above surjection $
M_{g_{1},n_{1}+d_{1}}\times \cdots \times M_{g_{k},n_{k}+d_{k}}\rightarrow M^{\mathrm{rs}}(G'_{\preceq G},w')
$ is clearly a finite map, we thus have 
\[
\dim M^{\mathrm{rs}}(G'_{\preceq G},w')=\sum_{i=1}^{k}(3g_{i}-3+n_{i}+d_{i})=3\sum_{i=1}^{k}g_{i}-3k+n+2d
\]
since $\sum_{i=1}^{k}d_{i}=2d$. Applying $g=\sum_{i=1}^{k}g_{i}+d-k+1$, we get  
\[
\dim M^{\mathrm{rs}}(G'_{\preceq G},w')=3g-3d+3k-3-3k+n+2d=3g-3+n-d=|E(G')|.
\]
\end{proof}

By the discussion in the beginning of this subsection, the tropicalization of a nodal Riemann surface is contracted from 
a trivalent genus $g$ graph with $n$ leaves (via the tropicalization of its normalization), 
so we have also a stratification for $\overline{M}_{g,n}$ as follows:
\[
\overline{M}_{g,n}=\bigsqcup_{\big(\text{$(G'_{\preceq G},w')$, $G$ trivalent, genus $g$, $n$ leaves}\big)/\sim}
M^{\mathrm{rs}}(G'_{\preceq G},w'),
\]
where $(G'_{\preceq G_{1}},w')\sim(G''_{\preceq G_{2}},w'')$ if they support the Riemann surfaces of the same topological type.

Then we have the following partition analogy between $\overline{M^{\mathrm{tr}}(G,\underline{0})}$ and $\overline{M}_{g,n}$.

\begin{theorem}\label{thm:analogy}
Fix a regular tropicalization $(G,\underline{0})$ of a smooth $n$-pointed genus $g$ Riemann surface. 
Let $(G',w')$ be a weighted contraction of 
it. Then the association as below
\[
M^{\mathrm{tr}}(G',w')\mapsto M^{\mathrm{rs}}(G'_{\preceq G},w')
\]
gives a map from the stratification of $\overline{M^{\mathrm{tr}}(G,\underline{0})}$ to the stratification of $\overline{M}_{g,n}$.
And we have 
\begin{enumerate}
  \item $\dim M^{\mathrm{tr}}(G',w')=\dim M^{\mathrm{rs}}(G'_{\preceq G},w')=|E(G')|$.
  
  \item Suppose $(G'',w'')\preceq (G',w')$, then we have $M^{\mathrm{tr}}(G'',w'')\subset  \overline{M^{\mathrm{tr}}(G',w')}$ 
  and $M^{\mathrm{rs}}(G''_{\preceq G},w'')\subset  \overline{M^{\mathrm{rs}}(G'_{\preceq G},w')}$.
\end{enumerate}
\end{theorem}

\begin{proof}
That $\dim M^{\mathrm{rs}}(G'_{\preceq G},w')=|E(G')|$ follows from Lemma $\ref{lem:irre-and-dim}$. The dimension of 
$M^{\mathrm{tr}}(G',w')$ comes from the dimension of the cone $\mathcal{C}(G',w')$ which is equal to $|E(G')|$. 
This completes the proof of the statement (1).

For (2), suppose $(G'',w'')\preceq (G',w')$, then we have $M^{\mathrm{tr}}(G'',w'')\subset  \overline{M^{\mathrm{tr}}(G',w')}$ 
by Proposition $\ref{prop:decomposition}$. Now we will show that we also have 
$M^{\mathrm{rs}}(G''_{\preceq G},w'')\subset  \overline{M^{\mathrm{rs}}(G'_{\preceq G},w')}$.

Let $(C,\underline{p})\in M^{\mathrm{rs}}(G''_{\preceq G},w'')$ be a pointed Riemann surface. We know that a neighborhood $U$ of 
each node on $C$ can be locally modeled as $((\mathbb{C}^{2},0),z_{1}z_{2}=0)$. We can take the neighborhood $U$ as small as 
we like, then a deformation of the node is modeled by a map germ 
$\iota:(\mathbb{C}^{2},0)\rightarrow (\mathbb{C},0);(z_{1},z_{2})\mapsto z_{1}z_{2}$, regarding the fiber $U_{t}:=\iota^{-1}(t)$ 
for small $t\neq 0$ as a `deformed $U$'.

Now let $(G'',w'')=(G'_{/Q},w'_{/Q})$. We can deform all the nodes corresponding to the edges in $Q$ at one time, so that 
we have a family of pointed Riemann surfaces supported on 
$(G'_{\preceq G},w')$ specializing to $(C,\underline{p})$. Namely, there exists such a family, $\pi:\mathscr{C}\rightarrow B$, 
over a one dimensional complex base manifold $(B,o)$, 
with $n$ pairwise disjoint holomorphic sections $(\sigma_{i}:B\rightarrow \mathscr{C})$ indicating the marked points, 
such that the fiber 
$(C_{b}:=\pi^{-1}(b),\underline{\sigma}(b))$ for every $b\neq o$ can be tropicalized to $(G'_{\preceq G},w')$ and 
the central fiber $(C_{o},\underline{\sigma}(o))=(C,\underline{p})$.
Hence we have $M^{\mathrm{rs}}(G''_{\preceq G},w'')\subset  \overline{M^{\mathrm{rs}}(G'_{\preceq G},w')}$.
\end{proof}

This theorem shows that there exists an order-preserving partial correspondence between the stratification of our moduli space 
$\overline{M^{\mathrm{tr}}(G,\underline{0})}$ and $\overline{M}_{g,n}$, not in a bijective manner though.

\subsection{Comparing with the moduli space of tropical curves}

For the moduli space of tropical curves, we only consider stable tropical curves.

\begin{definition}
A \emph{stable $n$-pointed weighted tropical curve of genus $g$} is a triple $(G,w,\ell)$ where $(G,w)$ is a 
stable weighted graph with a length function $\ell:E(G)\cup L(G)\rightarrow \mathbb{R}_{>0}\cup\{\infty\}$ 
such that $\ell(x)=\infty$ if and only if $x\in L(G)$.
\end{definition}

One readily notice that our tropicalizations for nodal Riemann surfaces with $2-2g-n<0$ are always stable 
tropical curves. So we will omit ``stable'' in what follows if no confusion would arise.

In order to construct the moduli space of tropical curves, in \cite{Caporaso-2013}, Caporaso first constructed the moduli 
space of tropical curves with fixed combinatorial type: $M^{\mathrm{trop}}(G,w)^{+}:=\mathcal{C}(G,w)^{+}/\sim$, where 
$p_{1}\sim p_{2}$ if the associated weighted metric graphs $(G_{1},w_{1},\ell_{1})$ and $(G_{2},w_{2},\ell_{2})$ are isomorphic to 
each other. Then the moduli space of $n$-pointed genus $g$ tropical curves is defined as 
\[
M_{g,n}^{\mathrm{trop}}:=\Big(\bigsqcup_{\text{$G$ trivalent, genus $g$, $n$ leaves}}M^{\mathrm{trop}}(G,\underline{0})^{+}\Big)/\cong
\]
where $\cong$ denotes the isomorphism of tropical curves. Finally this space can be compactified to 
\[
\overline{M_{g,n}^{\mathrm{trop}}}:=\Big(\bigsqcup_{\text{$G$ trivalent, genus $g$, $n$ 
		 leaves}}\overline{M^{\mathrm{trop}}(G,\underline{0})^{+}}\Big)/\cong
\]
by allowing the edge-length to go to infinity (so that the tropical curve becomes a so-called extended tropical curve), 
where $\cong$ denotes the isomorphism of extended tropical curves.

It can be shown that the topological space $\overline{M_{g,n}^{\mathrm{trop}}}$ is compact and Hausdorff and of pure dimension 
$3g-3+n$. Moreover, it is connected through codimension one, which is a typical property of a tropical variety, although 
$\overline{M_{g,n}^{\mathrm{trop}}}$ is not a tropical variety in general.

Without surprise, one may notice that the construction for our moduli space of tropicalizations 
$\overline{M^{\mathrm{tr}}(G,\underline{0})}$
is also largely based on the techniques of the moduli of weighted metric graphs. However, there are still at least two main 
differences between 
$\overline{M^{\mathrm{tr}}(G,\underline{0})}$ and $\overline{M_{g,n}^{\mathrm{trop}}}$.

(1) The topology is different: For $\overline{M^{\mathrm{tr}}(G,\underline{0})}$ 
we also allow the edge-length to go to infinity, but we identify the edge-length $\infty$ to $0$ 
so as to get a topological space $(\mathbb{R}_{\geq 0}\cup\{\infty\})/\sim$, homeomorphic to $\mathbb{S}^{1}$, 
for each coordinate to start with. 
That is because we want to regard our graphs as a tropicalization of some connected genus $g$ Riemann surface. When the 
length of some edge	goes to infinity, the Riemann surface supported on it becomes the normalization of 
the nodal Riemann surface supported on the graph 
whose corresponding edge-length goes to zero. Thus the Riemann surface becomes disconnected or of less genus, which 
is not the one we want to tropicalize. 
	
While for $\overline{M_{g,n}^{\mathrm{trop}}}$, we use the topological space $\mathbb{R}_{\geq 0}\cup \{\infty\}$, 
endowed with the subspace topology of the one-point compactification $\mathbb{R}\cup \{\infty\}$ of $\mathbb{R}$, 
for each coordinate to start with.

(2) The ``irreducibility'' is different: 
the regular tropicalization for a smooth $n$-pointed Riemann surface is not unique, and on the other hand, 
a weighted graph may support Riemann surfaces of distinct topological type if they are contracted from different 
regular tropicalizations, so we fix a regular tropicalization $(G,\underline{0})$ for constructing our space 
$\overline{M^{\mathrm{tr}}(G,\underline{0})}$. 

While for $\overline{M_{g,n}^{\mathrm{trop}}}$, it is glued by the moduli spaces 
of all the possible combinatorial types $\overline{M^{\mathrm{trop}}(G,\underline{0})^{+}}$ along the codimension one strata, so that 
each $\overline{M^{\mathrm{trop}}(G,\underline{0})^{+}}$ plays a role like an ``irreducible component'' in 
$\overline{M_{g,n}^{\mathrm{trop}}}$. Hence in this sense, our $\overline{M^{\mathrm{tr}}(G,\underline{0})}$ 
is just ``irreducible''.

\bibliographystyle{alpha}
\bibliography{tropbibfile}
\end{document}